\documentclass[12pt]{amsart}
\usepackage[utf8]{inputenc}
\usepackage[T1]{fontenc}
\usepackage{amsfonts}
\usepackage[leqno]{amsmath}
\usepackage{amssymb, comment, float}
\usepackage[dvipsnames]{xcolor}
\usepackage[foot]{amsaddr}
\usepackage[shortlabels]{enumitem}
\usepackage{mathdots}
\usepackage[footskip=0.5in,  headheight = 0.5in, top=1.25in, bottom=1.25in,  right=1in,  left=1in]{geometry}
\usepackage{hyperref}
\hypersetup{
colorlinks,
linkcolor=black,
urlcolor=Blue,
citecolor=black,}
\usepackage[center]{caption}
\usepackage{eufrak}
\usepackage{marvosym}     
\usepackage{latexsym}
\usepackage[nodayofweek]{datetime}
\usepackage{tikz}
\usepackage{subfig}
\usepackage{thm-restate}

\newcommand{\mca}{\mathcal}
\newcommand{\poi}{\mathbb{N}}

\newtheorem{theorem}{Theorem}[section]
\newtheorem{lemma}[theorem]{Lemma}

\usepackage[foot]{amsaddr}

\usepackage[T1]{fontenc}

\usepackage[leqno]{amsmath}

\makeatletter
\newcommand{\leqnomode}{\tagsleft@true}
\newcommand{\reqnomode}{\tagsleft@false}
\makeatother

\usepackage{latexsym}
\usepackage{amsfonts}

\usepackage{tikz}
\usepackage{float}
\usepackage{lmodern}

\def\dd{\hbox{-}}

\usepackage{marvosym}               

\DeclareMathOperator{\tw}{tw}

\newcounter{tbox}

\newcommand{\sta}[1]{\medskip\medskip\refstepcounter{tbox}\noindent{\parbox{\textwidth}{(\thetbox) \emph{#1}}}\vspace*{0.3cm}}

\newcommand{\mylongtitle}[1]{%
  \ifodd\value{page}%
    \protect\parbox{0.97\linewidth}{#1}\hfill%
  \else%
    \hfill\protect\parbox{0.97\linewidth}{#1}%
  \fi%
}

\newcommand{\mf}{\mathfrak}

\usepackage{latexsym}
\usepackage{amsfonts}
\usepackage[leqno]{amsmath}
\usepackage{tikz}
\usetikzlibrary{decorations.pathmorphing}
\usetikzlibrary{arrows,positioning}

\tikzset{snake it/.style={decorate, decoration=snake}}
\usepackage{float}
\usepackage{lmodern}
\usepackage[T1]{fontenc}

\def\dd{\hbox{-}}

\usepackage{marvosym}

\newcommand{\otherlabel}[2]{\protected@edef\@currentlabel{#2}\label{#1}}
\makeatother

\mathchardef\mh="2D

\title[Induced subgraphs and tree decompositions XII.]{Induced subgraphs and tree decompositions\\
XII. Grid theorem for pinched graphs}

\author{Bogdan Alecu$^{\ast \ast \mathparagraph}$}
\author{Maria Chudnovsky$^{\ast \amalg}$}
\author{Sepehr Hajebi$^{\mathsection}$}
\author{Sophie Spirkl$^{\mathsection \parallel}$}

\thanks{$^{\ast}$ Princeton University, Princeton, NJ, USA}
\thanks{$^{**}$ School of Computing, University of Leeds, Leeds, UK}
\thanks{$^{\mathsection}$ Department of Combinatorics and Optimization, University of Waterloo, Waterloo, Ontario, Canada}
\thanks{$^{\amalg}$ Supported by NSF-EPSRC Grant DMS-2120644 and by AFOSR grant FA9550-22-1-0083.}
\thanks{$^{\mathparagraph}$ Supported by DMS-EPSRC Grant EP/V002813/1.}
\thanks{$^{\parallel}$ We acknowledge the support of the Natural Sciences and Engineering Research Council of Canada (NSERC), [funding reference number RGPIN-2020-03912].
Cette recherche a \' et\' e financ\' ee par le Conseil de recherches en sciences naturelles et en g\' enie du Canada (CRSNG), [num\' ero de r\' ef\' erence RGPIN-2020-03912]. This project was funded in part by the Government of Ontario.}
\date {\today}

\begin{document}

\maketitle
\begin{abstract}
Given $c\in \poi$, we say a graph $G$ is \textit{$c$-pinched} if $G$ does not contain an induced subgraph consisting of $c$ cycles, all going through a single common vertex and otherwise pairwise disjoint and with no edges between them. What can be said about the structure of $c$-pinched graphs?

For instance, $1$-pinched graphs are exactly graphs of treewidth $1$. However, bounded treewidth for $c>1$ is immediately seen to be a false hope because complete graphs, complete bipartite graphs, subdivided walls and line graphs of subdivided walls are all examples of $2$-pinched graphs with arbitrarily large treewidth. There is even a fifth obstruction for larger values of $c$, discovered by Pohoata and later independently by Davies, consisting of $3$-pinched graphs with unbounded treewidth and no large induced subgraph isomorphic to any of the first four obstructions.

We fuse the above five examples into a grid-type theorem fully describing the unavoidable induced subgraphs of pinched graphs with large treewidth. More precisely, we prove that for \textit{every} $c\in \poi$, a $c$-pinched graph $G$ has large treewidth if and only if $G$ contains one of the following as an induced subgraph: a large complete graph, a large complete bipartite graph, a subdivision of a large wall, the line graph of a subdivision of a large wall, or a large graph from the Pohoata-Davies construction. Our main result also generalizes to an extension of pinched graphs where the lengths of excluded cycles are lower-bounded.
\end{abstract}

\section{Introduction}\label{sec:intro}

\subsection{Background}

The set of all positive integers is denoted by $\poi$. Graphs in this paper have finite vertex sets, no loops and no parallel edges. Let $G$ be a graph. For $X \subseteq V(G)$, we denote by $G[X]$ the subgraph of $G$ induced by $X$, and by $G \setminus X$ the induced subgraph of $G$ obtained by removing $X$. We use induced subgraphs and their vertex sets interchangeably. For graphs $G$ and $H$, we say $G$ \emph{contains} $H$ if $G$ has an induced subgraph isomorphic to $H$, and we say $G$ is \emph{$H$-free} if $G$ does not contain $H$. A class of graphs is \textit{hereditary} if it is closed under isomorphism and taking induced subgraphs.

The \textit{treewidth} of a graph $G$ (denoted by $\tw(G)$) is the smallest $w\in \poi$ for which there exists a tree $T$ as well as an assignment $(T_v:v\in V(G))$ of non-empty subtrees of $T$ to the vertices of $G$ with the following specifications.
\begin{enumerate}[(T1), leftmargin=15mm, rightmargin=7mm]
\item For every edge $uv\in V(G)$, $T_u$ and $T_v$ share at least one vertex.

\item For every $x\in V(T)$, there are at most $w+1$ vertices $v\in V(G)$ for which $x\in V(T_v)$.
\end{enumerate}

As one of the most extensively studied graph invariants, the enduring interest in treewidth is partly explained by its role in the development of Robertson and Seymour's graph minors project, as well as the vast range of nice structural \cite{RS-GMV} and algorithmic \cite{Bodlaender1988DynamicTreewidth} properties of graphs of small treewidth.

Graphs of large treewidth have also been a central topic of research for several decades. Usually, it is most desirable to certify large treewidth in a graph $G$ by means of a well-understood ``obstruction'' which still has relatively large treewidth, and which lies in $G$ under a certain containment relation. The cornerstone result in this category is the so-called \textit{Grid Theorem} of Robertson and Seymour \cite{RS-GMV}, Theorem~\ref{wallminor} below, which says that under two of the most studied graph containment relations, namely the graph minor relation and the subgraph relation, the only obstructions to bounded treewidth are the ``basic'' ones: the $t$-by-$t$ square grid for minors, and subdivisions of the $t$-by-$t$ hexagonal grid for subgraphs. The $t$-by-$t$ hexagonal grid is also known as the \textit{$t$-by-$t$ wall}, denoted by $W_{t\times t}$ (see Figure~\ref{fig:Grid+Wall}, and also \cite{wallpaper} for full definitions).

\begin{theorem}[Robertson and Seymour \cite{RS-GMV}]\label{wallminor}
For every $t\in \poi$,
every graph of sufficiently large
treewidth contains the $t$-by-$t$ square grid as a minor, or equivalently, a subdivision of $W_{t\times t}$ as a subgraph.
\end{theorem}

\begin{figure}[t!]
\centering
\includegraphics[scale=0.7]{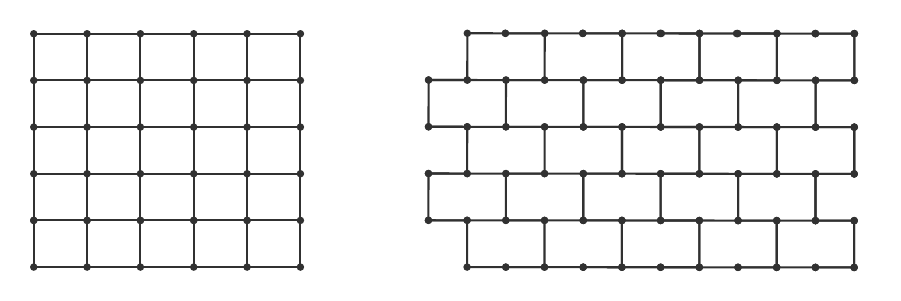}
\caption{The $6$-by-$6$ square grid (left) and the $6$-by-$6$ wall $W_{6\times 6}$ (right).}
\label{fig:Grid+Wall}
\end{figure}

It is therefore tempting to inquire about an analogue of Theorem~\ref{wallminor} for another standard graph containment relation: induced subgraphs. The basic obstructions in this case already suggest that a more involved grid-type theorem is to be expected: complete graphs, complete bipartite graphs, subdivided walls, and line graphs of subdivided walls are all examples of induced-subgraph-minimal graphs with large treewidth. It is convenient to group all these graphs together. Given $t\in \poi$, we say a graph $H$ is a \textit{$t$-basic obstruction} if $H$ is isomorphic to one of the following: the complete graph $K_t$, the complete bipartite graph $K_{t,t}$, a subdivision of $W_{t\times t}$, or the line graph of a subdivision of $W_{t\times t}$, where the {\em line graph} $L(F)$ of a graph $F$ is the graph with vertex set $E(F)$, such that two vertices of $L(F)$ are adjacent if and only if the corresponding edges of $F$ share an end (see Figure~\ref{fig:3basic}). We say a graph $G$ is \textit{$t$-clean} if $G$ does not contain a $t$-basic obstruction (as an induced subgraph). A graph class $\mca{G}$ is \textit{clean} if for every $t\in \poi$, there is a constant $w(t)\in \poi$ (depending on $\mca{G}$) for which every $t$-clean graph in $\mca{G}$ has treewidth at most $w(t)$. Since the basic obstructions have unbounded treewidth ($K_{t+1}$, $K_{t,t}$, subdivisions of $W_{t\times t}$ and line graphs of subdivisions of $W_{t\times t}$ are all known to have treewidth $t$), it follows that for every hereditary class of bounded treewidth, there exists some $t\in \poi$ such that every graph in the class is $t$-clean. The converse would be a particularly nice grid-type theorem for induced subgraphs: \textit{every hereditary class is clean.} This, however, is now known to be far from true, thanks to the numerous constructions \cite{deathstar, Davies2, Pohoata, layered-wheels} of graphs with arbitrarily large treewidth which are $t$-clean for small values of $t$ (and we will take a closer look at the one from \cite{Davies2, Pohoata} in a moment).

On the other hand, there are several hereditary classes that are known to be clean for highly non-trivial reasons. As a notable example, Korhonen \cite{Korhonen} proved that every graph class of bounded maximum degree is clean, settling a conjecture from \cite{aboulker}:
\begin{theorem}[Korhonen \cite{Korhonen}]\label{thm:korhonen}
For every $d\in \poi$, the class of all graphs with maximum degree at most $d$ is clean.
\end{theorem}

\begin{figure}[t!]
\centering
\includegraphics[scale=0.65]{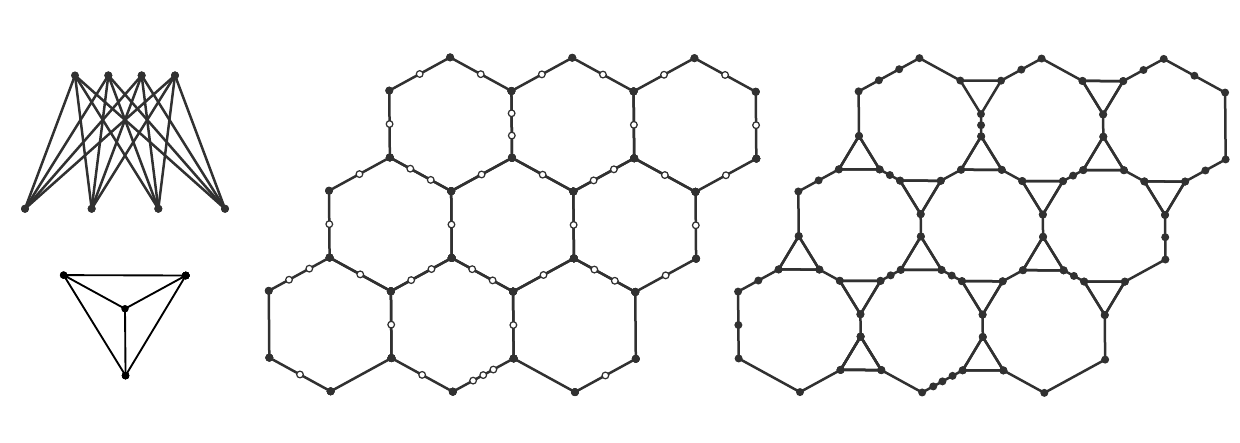}
\caption{The $4$-basic obstructions.}
\label{fig:3basic}
\end{figure}

One possible attempt at generalizing Theorem~\ref{thm:korhonen} is to look for clean classes under weaker assumptions than bounded maximum degree. For instance, bounded maximum degree is equivalent to excluding a fixed star as a subgraph, and in a recent joint work with Abrishami, we extended Theorem~\ref{thm:korhonen} to graphs that exclude a fixed subdivided star as an induced subgraph. In fact, we proved:

\begin{theorem}[Abrishami, Alecu, Chudnovsky, Hajebi, Spirkl \cite{twvii}]\label{tw7}
Let $H$ be a graph. Then the class of all $H$-free graphs is clean if and only if every component of $H$ is a subdivided star.
\end{theorem}
Another natural candidate for a configuration forcing large-degree vertices consists of several cycles sharing a single vertex. For $c\in \poi$, let us say a graph $G$ is \textit{$c$-pinched} if $G$ does not contain $c$ induced cycles, all going through a common vertex and otherwise pairwise disjoint and anticomplete (for disjoint subsets $X,Y$ of vertices in a graph $G$, we say that $X$ is \emph{anticomplete}
to $Y$ if no edges between $X$ and $Y$ are present in $G$, and that $X$ is \textit{complete} to $Y$ if all edges with an end in $X$ and an end in $Y$ are present in $G$). Note that $1$-pinched graphs are forests, which are the only graphs with treewidth $1$. For $c=2$, it is easily seen that all basic obstructions are $2$-pinched, and we will show that they are the only representatives of large treewidth in $2$-pinched graphs:

\begin{theorem}\label{thm:2pinched}
The class of all $2$-pinched graphs is clean.
\end{theorem}

One may then ask if the class of $c$-pinched graphs is clean for all $c\in \poi$. But this is false, as shown by a 10-year-old construction due to Pohoata \cite{Pohoata}, also re-discovered recently by Davies \cite{Davies2}, which we describe below.

Given an integer $k$, we write $\poi_k$ to denote the set of all positive integers less than or equal to $k$ (so we have $\poi_k=\varnothing$ if and only if $k\leq 0$). For $s\in \poi$, let $PD_s$ be the graph whose vertex set can be partitioned into a stable set $S=\{x_1,\ldots, x_s\}$, and $s$ pairwise disjoint induced paths $L_1,\ldots, L_s$ with no edges between them, such that the following hold.
\begin{enumerate}[(PD1), leftmargin=19mm, rightmargin=7mm]
\item\label{PD1} For every $i\in \poi_s$, $L_i$ has length $s-1$ (and so exactly $s$ vertices).
\item \label{PD2} For every $i\in \poi_s$, the vertices in the interior of $L_i$ may be enumerated from one end to the other as $u_{1}^i\dd \cdots\dd u_s^i$ such that for every $j\in \poi_s$, $x_j$ has \textbf{exactly} one neighbor in $V(L_i)$, namely $u_j^i$.
\end{enumerate}
\begin{figure}[t!]
\centering
\includegraphics[scale=0.75]{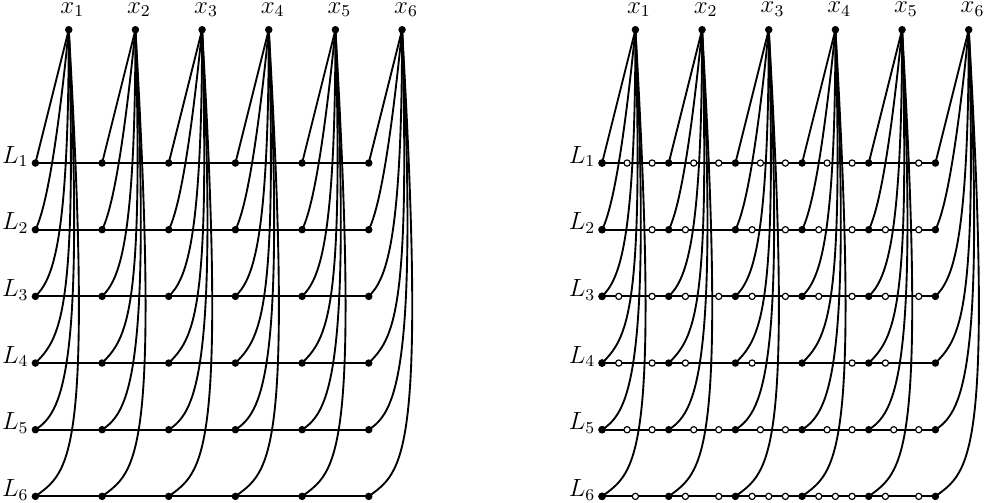}
\captionof{figure}{The graph $PD_6$ (left) and an expansion of $PD_6$ (right).}
\label{fig:davies}
\end{figure}
See Figure~\ref{fig:davies}. It is straightforward to show that for every $s\in \poi$, $PD_s$ is a $4$-clean, $3$-pinched graph of treewidth at least $s$. More generally, as we will prove in Theorem~\ref{thm:Oproperties}, the same holds for every graph obtained from $PD_s$ by subdividing the edges of $L_1,\ldots, L_s$ arbitrarily. We refer to these graphs as \textit{expansions} of $PD_s$ (see Figure~\ref{fig:davies}).
It follows that expansions of the Pohoata-Davies graphs are ``non-basic'' obstructions to bounded treewidth in pinched graphs. Strikingly, the converse turns out to hold, too. We prove that:

\begin{theorem}\label{mainthm}
For all $c,s,t\in \poi$, every $c$-pinched graph of sufficiently large treewidth contains either a $t$-basic obstruction or an expansion of $PD_s$.
\end{theorem}
Since the basic obstructions and the graphs $PD_s$ have arbitrarily large treewidth, Theorem~\ref{mainthm} provides a full grid-type theorem for the class of $c$-pinched graphs for all $c\in \poi$. More generally, our main result in this paper, Theorem~\ref{mainthmgeneral}, renders a complete description of the induced subgraph obstructions to bounded treewidth in the class of \textit{$(c,h)$-pinched} graphs for all $c,h\in \poi$, that is, graphs containing no $c$ induced cycles each of length at least $h+2$, all going through a common vertex and otherwise pairwise disjoint and anticomplete (so a graph is $c$-pinched if and only if it is $(c,1)$-pinched). Indeed, the strengthening is direct enough that Theorem~\ref{mainthm} is the special case of Theorem~\ref{mainthmgeneral} where $h=1$. Note also that no expansion of the graph $PD_{4}$ is $2$-pinched. Therefore, Theorem~\ref{mainthm} implies Theorem~\ref{thm:2pinched}.

Let us remark that grid-type  theorems involving non-basic obstructions, such as Theorem~\ref{mainthm} (or Theorem  \ref{mainthmgeneral}, rather), are as of yet quite rare. Indeed, the only other examples we are aware of are the analogous result for ``$c$-perforated'' graphs -- which we proved recently \cite{twix} --
and an result from \cite{woodcirclegraphs} concerning the class of ``circle graphs.'' In fact, the proof of Theorem~\ref{mainthmgeneral} bears a close resemblance to that of the main result of \cite{twix}, and crucially builds on some tools developed there.

To elaborate, let $\mca{G}$ be the non-clean class for which we wish to prove a grid-type theorem. Roughly speaking, the idea is to break the proof into two steps: first, we show that every $t$-clean graph in the class with sufficiently large treewidth must contain an ``approximate version'' of the non-basic obstruction we are looking for, and second, we perform further analysis on the approximate version in pursuit of the exact one. Luckily, in the case of pinched graphs, the approximate and the exact non-basic obstructions are actually quite close. Note that the (expansions of) Pohoata-Davies graphs consist of a stable set and a number of pairwise disjoint and anticomplete paths such that every vertex in the stable set has a neighbor in every path. We call such a configuration in a graph $G$ a ``constellation'' (a notion also used in \cite{twix} as an approximate non-basic obstruction for perforated graphs).

As for the present paper, our first goal is to show that for all $c,t,h\in \poi$, every $t$-clean $(c,h)$-pinched graph of sufficiently large treewidth contains a huge constellation. This involves a useful result from an earlier paper in this series \cite{twvii} concerning the ``local connectivity'' in clean classes, accompanied by a collection of Ramsey-type arguments to tidy up an induced subgraph of $G$ with high local connectivity. The second step then is to turn a constellation into an expansion of a Pohoata-Davies-like graph. To that end, for every path $L$ in the ``path side'' of the constellation, we consider the intersection graph $I$ of the minimal subpaths of $L$ containing all neighbors of each vertex in the ``stable set side'' $S$. Provided that $S$ is large enough, $I$ contains either a big stable set or a big clique. In the former case, on $L$, the neighbors of the vertices from the stable set do not interlace, and the resulting ``alignment'' of vertices according to their neighbors on $L$ signals the emergence of a Pohoata-Davies-like structure. In the latter case, several vertices in $S$ turn out to have neighbors in several pairwise disjoint and anticomplete subpaths of $L$. This eventually yields $c$ induced cycles with a vertex in common and otherwise pairwise disjoint and anticomplete, a contradiction.

We take the two steps above in the reverse order in Sections~\ref{sec:deal} and \ref{sec:obtain}, respectively. In the next section, we discuss the connectivity result from \cite{twvii} (together with its bells and whistles). Section~\ref{sec:ABC} introduces a variety of notions from \cite{twix} which we use in this paper, and also features the statement of our main result, Theorem~\ref{mainthmgeneral}, of which we give a complete proof in Section~\ref{sec:end}.

\section{Blocks}\label{sec:blocks}

We begin with a couple of definitions. Let $G = (V(G),E(G))$ be a graph. For an induced subgraph $H$ of $G$ and a vertex $x\in V(G)$, we denote by $N_H(x)$ the set of all neighbors of $x$ in $H$, and write $N_H[x]=N_H(x)\cup \{x\}$. A \textit{stable set in $G$} is a set of pairwise non-adjacent vertices. A {\em path in $G$} is an induced subgraph of $G$ which is a path. If $P$ is a path in $G$, we write $P = p_1 \dd \cdots \dd p_k$ to mean that $V(P) = \{p_1, \dots, p_k\}$ and $p_i$ is adjacent to $p_j$ if and only if $|i-j| = 1$. We call the vertices $p_1$ and $p_k$ the \emph{ends of $P$} and write $\partial P=\{p_1,p_k\}$. The \emph{interior of $P$}, denoted $P^*$, is the set $P \setminus \partial P$. For $x,y\in V(P)$, we denote by $x\dd P\dd y$ the subpath of $P$ with ends $x,y$. The \textit{length} of a path is its number of edges. Similarly, a {\em cycle in $G$} is an induced subgraph of $G$ that is a cycle. If $C$ is a cycle in $G$, we write $C= c_1 \dd \cdots \dd c_k\dd c_1$ to mean that $V(C) = \{c_1, \dots, c_k\}$ and $c_i$ is adjacent to $c_j$ if and only if $|i-j|\in \{1,k-1\}$. The \textit{length} of a cycle is also its number of edges. For a collection $\mca{P}$ of paths in $G$, we adopt the notations $V(\mca{P})=\bigcup_{P\in \mca{P}}V(P)$, $\mca{P}^*=\bigcup_{P\in \mca{P}}P^*$ and $\partial \mca{P}=\bigcup_{P\in \mca{P}}\partial P$.

Let $k\in \poi$ and let $G$ be a graph. A \textit{$k$-block} in $G$ is a pair $(B, \mca{P})$ where $B\subseteq V(G)$ with $|B|\geq k$ and $\mca{P}:{B\choose 2}\rightarrow 2^{V(G)}$ is map such that $\mca{P}_{\{x,y\}}=\mca{P}(\{x,y\})$, for each $2$-subset $\{x,y\}$ of $B$, is a set of at least $k$ pairwise internally disjoint paths in $G$ from $x$ to $y$. We say that $(B,\mca{P})$ is \textit{strong} if for all distinct $2$-subsets $\{x,y\}, \{x',y'\}$ of $B$, we have $V(\mca{P}_{\{x,y\}})\cap V(\mca{P}_{\{x',y'\}})=\{x,y\}\cap\{x',y'\}$; that is, each path $P\in \mca{P}_{\{x,y\}}$ is disjoint from each path $P'\in \mca{P}_{\{x',y'\}}$, except $P$ and $P'$ may share an end. In \cite{twvii}, with Abrishami we proved the following:

\begin{theorem}[Abrishami, Alecu, Chudnovsky, Hajebi, Spirkl \cite{twvii}]\label{noblocksmalltw_wall}
For all $k,t\in \poi$, there is a constant $\xi=\xi(k,t)\in \poi$ such that for every $t$-clean graph $G$ of treewidth more than $\xi$, there is a strong $k$-block in $G$.
\end{theorem}

The following result, which guarantees that the graphs we work with exclude ``short subdivisions'' of large complete graphs, paves the way for our application of Theorem~\ref{noblocksmalltw_wall}. Recall that a \textit{subdivision} of graph $H$ is a graph $H'$ obtained from $H$ by replacing the edges of $H$ with
pairwise internally disjoint paths of non-zero length between the corresponding ends. Let $r\in \poi\cup \{0\}$. An \textit{$(\leq r)$-subdivision} of $H$ is a subdivision of $H$
in which the path replacing each edge has length at most $r+1$.
\begin{theorem}[Dvo\v{r}\'{a}k, see Theorem 6 in \cite{dvorak}; Lozin and Razgon, see Theorem 3 in \cite{lozin}]\label{dvorak}
For every graph $H$ and all $d\in \poi\cup \{0\}$ and $t\in \poi$, there is a constant $m=m(H,d,t)\in \poi$ with the following property. Let $G$ be a graph with no induced subgraph isomorphic to a subdivision of $H$. Assume that $G$ contains a $(\leq d)$-subdivision of $K_m$ as a subgraph. Then $G$ contains either $K_t$ or $K_{t,t}$.
\end{theorem}

\section{Bundles and constellations}\label{sec:ABC}

In this section we state our main result, Theorem~\ref{mainthmgeneral}. We start with a few definitions that first appeared in \cite{twix}.

Let $G$ be a graph and let $l\in \poi$. By an \textit{$l$-polypath in $G$} we mean a set $\mca{L}$ of $l$ pairwise disjoint paths in $G$. We say $\mca{L}$ is \textit{plain} if every two distinct paths $L,L'\in \mca{L}$ are anticomplete in $G$.  Also, two polypaths $\mca{L}$ and $\mca{L}'$ in $G$ are said to be \textit{disentangled} if $V(\mca{L})\cap V(\mca{L}')=\varnothing$. For $s\in \poi$, an \textit{$(s,l)$-bundle in $G$} is a pair $\mf{b}=(S_{\mf{b}},\mca{L}_{\mf{b}})$ where $S_{\mf{b}}\subseteq V(G)$ with $|S_{\mf{b}}|=s$ and $\mca{L}_{\mf{b}}$ is an $l$-polypath in $G$ (note that $S_{\mf{b}}$ and $V(\mca{L}_{\mf{b}})$ are not necessarily disjoint). If $l=1$, say $\mca{L}_{\mf{b}}=\{L_{\mf{b}}\}$, we also denote the $(s,1)$-bundle $\mf{b}$ by the pair $(S_{\mf{b}},L_{\mf{b}})$.
Given an $(s,l)$-bundle $\mf{b}$ in $G$, we write $V(\mf{b})=S_{\mf{b}}\cup V(\mca{L}_{\mf{b}})$, and for every $L\in \mca{L}_{\mf{b}}$, we denote by $\mf{b}_L$ the $(s,1)$-bundle $(S_{\mf{b}},L)$. Also, we say that $\mf{b}$ is \textit{plain} if the $l$-polypath $\mca{L}_{\mf{b}}$ is plain. For two bundles $\mf{b}$ and $\mf{b}'$ in a graph $G$, we say  $\mf{b}$ and $\mf{b}'$ are \textit{disentangled} if $V(\mf{b})\cap V(\mf{b}')=\varnothing$. 

An \textit{$(s,l)$-constellation in $G$} is an $(s,l)$-bundle $\mf{c}=(S_{\mf{c}},\mca{L}_{\mf{c}})$ in $G$ such that $S_{\mf{c}}$ is a stable set (of cardinality $s$) in $G\setminus V(\mca{L}_{\mf{c}})$, and every $s\in S_{\mf{c}}$ has a neighbor in every path $L\in \mca{L}_{\mf{c}}$.

Let $\mf{c}$ be an $(s,1)$-constellation in a graph $G$. For a vertex $x\in S_{\mf{c}}$, by an \textit{$x$-gap in $\mf{c}$} we mean a path $P$ in $L_{\mf{c}}$ (possibly of length zero) where $x$ is adjacent to the ends of $P$ and anticomplete to $P^*$. For $d\in \poi$, we say $\mf{c}$ is \textit{$d$-hollow} if for every $x\in S_{\mf{c}}$, every $x$-gap in $\mf{c}$ has length less than $d$. Also, we say $\mf{c}$ is \textit{$d$-meager} if every vertex in $L_{\mf{c}}$ is adjacent to at most $d$ vertices in $S_{\mf{c}}$ (see Figure~\ref{fig:asterism}). In general, for an $(s,l)$-constellation $\mf{c}$, we say $\mf{c}$ is \textit{$d$-hollow} (\textit{$d$-meager}) if for every $L\in \mca{L}_{\mf{c}}$, $\mf{c}_L$ is $d$-hollow ($d$-meager).

For $s\in \poi$, an \textit{$s$-alignment} in $G$ is a triple $(S,L,\pi)$ where $(S,L)$ is an $(s,1)$-constellation in $G$ and $\pi: \poi_s\rightarrow S$ is a bijection such that for some end $u$ of $L$, the following holds.
 \begin{enumerate}[(AL), leftmargin=19mm, rightmargin=7mm]
     \item\label{AL} For all $i,j\in \poi_s$ with $i<j$, every neighbor $v_i\in L$ of $\pi(i)$ and every neighbor $v_j\in L$ of $\pi(j)$, the path in $L$ from $u$ to $v_j$ contains $v_i$ in its interior. In other words, traversing $L$ starting at $u$, all neighbors of $\pi(i)$ appear before all neighbors of $\pi(j)$.
 \end{enumerate}
See Figure~\ref{fig:asterism}. In particular, $(S,L)$ is $1$-meager.

\begin{figure}[t!]
  \centering
  \includegraphics[scale=0.8]{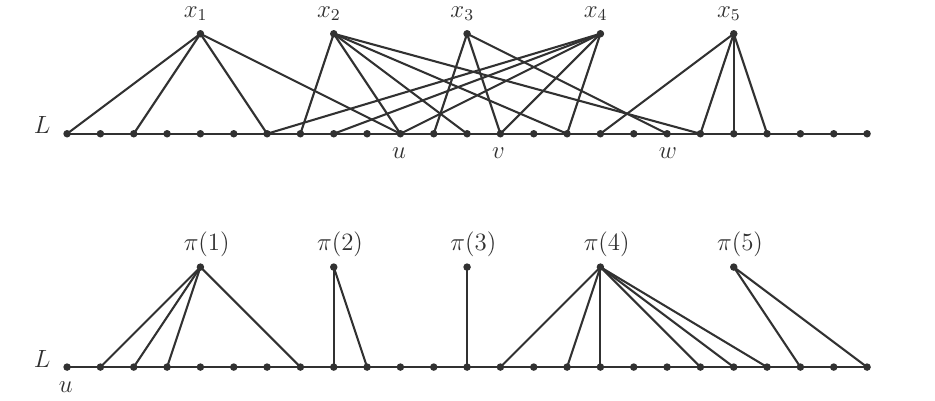}
  \captionof{figure}{Top: a $(5,1)$-constellation $\mf{c}$ with $S_\mf{c}=\{x_1,x_2,x_3,x_4,x_5\}$ and $L_{\mf{c}}=L$. Note that $\mf{c}$ is $3$-meager (with $u$ being the only vertex in $L$ with three neighbors in $S_{\mf{c}}$) and $6$-hollow (with the $x_3$-gap $v\dd L\dd w$ of length five being the longest). Bottom: a $5$-alignment.}
  \label{fig:asterism}
  \end{figure}

For $h,s\in \poi$, an \textit{$(s,h)$-array} in a graph $G$ is a plain, $h$-hollow $(s,s)$-constellation $\mf{a}$ in $G$ which satisfies the following.
 \begin{enumerate}[(AR), leftmargin=19mm, rightmargin=7mm]
     \item\label{AR} There exists a bijection $\pi:\poi_s\rightarrow S_{\mf{a}}$ such that for every $L\in \mca{L}_{\mf{a}}$, $(S_{\mf{a}},L,\pi)$ is an $s$-alignment in $G$.
 \end{enumerate}
See Figure~\ref{fig:array}. It is readily observed that for $X\subseteq V(G)$, if there is an $(s,1)$-array $\mf{a}$ in $G$ with $V(\mf{a})=X$, then $G[X]$ contains an expansion of the graph $PD_s$, and if $G[X]$ is isomorphic to an expansion of $PD_s$ for some $X\subseteq V(G)$, then there is an $(s,1)$-array $\mf{a}$ in $G$ with $V(\mf{a})=X$. Moreover, we have:

\begin{theorem}\label{thm:Oproperties}
   Let $h,s\in \poi$ and let $G$ be a graph. Let $\mf{a}$ be an $(s,h)$-array in $G$. Then $G[V(\mf{a})]$ is a $4$-clean, $(3,h)$-pinched graph of treewidth at least $s$.
\end{theorem}
\begin{proof}
 Let $J=G[V(\mf{a})]$. Note that $J$ contains $K_{s,s}$ as a minor (by contracting each path $L\in \mca{L}_{\mf{a}}$ into a vertex), which implies that $\tw(J)\geq s$. For the rest of the proof, let $\pi:\poi_s\rightarrow S_{\mf{a}}$ be the bijection satisfying \ref{AR}. 
 
Assume that $J$ has an induced subgraph isomorphic to a subdivision of $W_{4\times 4}$ or the line graph of a subdivision of $W_{4\times 4}$. Note that $W_{4\times 4}$ has an induced subgraph isomorphic to a subdivision of $K_4$ with exactly one unsubdivided edge (see Figure~\ref{fig:k4sub}). Thus, $J$ has an induced subgraph $W$ that is isomorphic to either a subdivision of $K_4$ with at most one unsubdivided edge or the line graph of a subdivision of $K_4$ with at most one unsubdivided edge. In either case, it is straightforward to observe that for every vertex $w\in V(W)$, the graph $W\setminus N_W[w]$ is connected. On the other hand, $W$ is an induced subgraph of $J$ with $\tw(W)\geq 3$ (because $W$ contains $K_4$ as a minor). Consequently, we have $|S_{\mf{a}}\cap V(W)|\geq 3$; say $w_1,w_2,w_3\in S_{\mf{a}}\cap V(W)$ such that $\pi(w_1)<\pi(w_2)<\pi(w_3)$. But now $W\setminus N_W[w_2]$ is disconnected, a contradiction. We deduce that $J$ has no induced subgraph isomorphic to a subdivision of $W_{4\times 4}$ or the line graph of a subdivision of $W_{4\times 4}$. Moreover, it is easy to check that $J$ is $K_4$-free and $K_{2,3}$-free. Hence, $J$ is $4$-clean.
\begin{figure}[t!]
    \centering
    \includegraphics[scale=0.6]{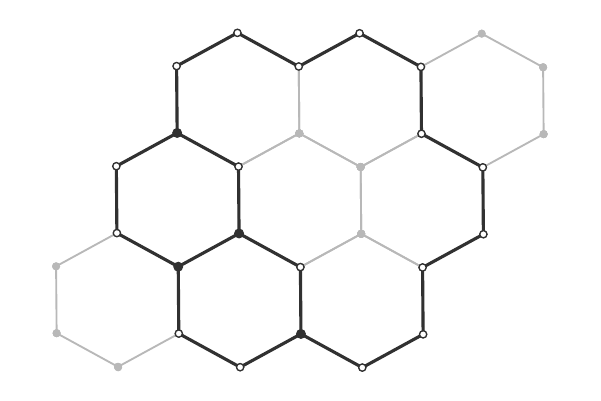}
    \caption{An induced subgraph of $W_{4\times 4}$ isomorphic to a subdivision of $K_4$ with exactly one unsubdivided edge.}
    \label{fig:k4sub}
\end{figure}

 It remains to show that $J$ is $(3,h)$-pinched. Suppose for a contradiction that there are three cycles $C_1,C_2,C_3$ in $J$ of length at least $h+2$ with $C_1\cap C_2\cap C_3=\{x\}$ and $C_1\setminus \{x\},C_2\setminus \{x\},C_3\setminus \{x\}$ are pairwise disjoint and anticomplete. Since $x$ has degree at least six in $J$, it follows that $x\in S_{\mf{a}}$. Also, since $\mf{a}$ is $h$-hollow, it follows that there is no cycle of length at least $h+2$ in $J[V(\mca{L}_{\mf{a}}\cup \{x\})]$. Thus, for every $l\in \{1,2,3\}$, we may pick $x_{l}\in C_{l}\cap (S_{\mf{a}}\setminus \{x\})\neq \varnothing$. By symmetry, we may assume that there are distinct $l_1,l_2\in \{1,2,3\}$ for which $\pi^{-1}(x)$ is smaller than $\pi^{-1}(x_{l_1})$ and $\pi^{-1}(x_{l_2})$. In particular, we may assume without loss of generality that $\pi^{-1}(x)< \pi^{-1}(x_1)< \pi^{-1}(x_2)$. But then $x,x_2\in C_2$ are in different components of $J\setminus N_J[x_1]$ (see Figure~\ref{fig:array}), which violates the fact that $C_1\setminus \{x\}$ and $C_2\setminus \{x\}$ are disjoint and anticomplete. This completes the proof of Theorem~\ref{thm:Oproperties}.
\end{proof}
\begin{figure}[t!]
  \centering
  \includegraphics[scale=0.6]{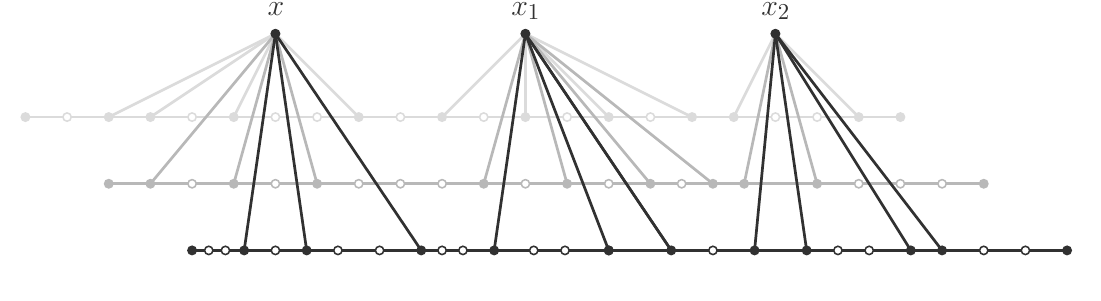}
  \captionof{figure}{A $(3,4)$-array (vertex labels $x,x_1$ and $x_2$ are relevant in the proof of Theorem~\ref{thm:Oproperties}).}
  \label{fig:array}
  \end{figure}
 Now we can state the main result of this paper:
\begin{theorem}\label{mainthmgeneral}
    For all $c,h,s,t\in \poi$, there is a constant $\tau=\tau(c,h,s,t)\in \poi$ such that for every $t$-clean $(c,h)$-pinched graph $G$ of treewidth more than $\tau$, there is an $(s,h)$-array in $G$.
\end{theorem}
In view of Theorem~\ref{thm:Oproperties}, Theorem~\ref{mainthmgeneral} yields, for all $c,h\in \poi$, a complete description of unavoidable induced subgraphs of $(c,h)$-pinched graphs with large treewidth. Moreover, as mentioned above, there is an $(s,1)$-array in a graph $G$ if and only if $G$ contains an expansion of $PD_s$. So Theorem~\ref{mainthm} is the special case of Theorem~\ref{mainthmgeneral} for $h=1$. We also point out that for $s\geq 4$ and $h\in \poi$, if a graph $G$ is $(2,h)$-pinched, then there is no $(s,h)$-array in $G$. Hence, Theorem~\ref{mainthm} implies a strengthening of Theorem~\ref{thm:2pinched}, that for every $h\in \poi$, the class of all $(2,h)$-pinched graphs is clean.

\section{Dealing with plain constellations}\label{sec:deal}

In this section, we prove the following:
\begin{theorem}\label{thm:dealwithconst} For all $c,h,s,t\in \poi$, there are constants $\sigma=\sigma(c,h,s,t)\in \poi$ and $\lambda=\lambda(c,h,s,t)\in \poi$ with the following properties. Let $G$ be a $(c,h)$-pinched graph which does not $K_t$ and $K_{t,t}$. Assume that there exists a plain $(\sigma,\lambda)$-constellation $\mf{c}$ in $G$. Then there is an $(s,h)$-array in $G$.
  \end{theorem}

We need a couple of lemmas, beginning with the following. Although we have proved a similar result in \cite{twix}, we include the proof here as it is short.

\begin{lemma}\label{lem:ast_to_alignment}
     Let $a,d,s,l\in \poi$ and let $G$ be a graph. Assume that there exists a $d$-meager $(a^{l-1}(s+d(l-1)),1)$-constellation $(S_0,L_0)$ in $G$. Then one of the following holds.
      \begin{enumerate}[\rm (a)]
        \item\label{lem:ast_to_alignment_a} There exists an $a$-alignment $(S,L,\pi)$ in $G$ with $S\subseteq S_0$ and $L\subseteq L_0$.
        \item\label{lem:ast_to_alignment_b} There exists a plain $(s,l)$-constellation $\mf{c}$ in $G$ such that $S_{\mf{c}}\subseteq S_0$ and $L\subseteq L_0$ for every $L\in \mca{L}_{\mf{c}}$. In particular, $\mf{c}$ is also $d$-meager.
      \end{enumerate}
\end{lemma}
  \begin{proof}For fixed $a,d$ and $s$, we proceed by induction on $l$. Note that if $l=1$,  then $(S_0,L_0)$ is an $(s,1)$-constellation in $G$ satisfying Lemma~\ref{lem:ast_to_alignment}\ref{lem:ast_to_alignment_b}. Thus, we may assume that $l\geq 2$.

  Let $u,v$ be the ends of $L_0$. For every vertex $x\in S_{0}$, traversing $L_0$ from $u$ to $v$, let $u_x$ and $v_x$ be the first and the last neighbor of $x$ in $L_0$ (possibly $u_x=v_x$), and let $L_x=u_x\dd L_0\dd v_x$. Let $Y$ be a subset of $S_0$ for which the paths $\{L_y:y\in Y\}$ are pairwise disjoint, such that $|Y|$ is as large as possible, and subject to this, $\sum_{y\in Y}|u\dd L_0\dd u_y|$ is as large as possible. Clearly, if $|Y|\geq a$, then Lemma~\ref{lem:ast_to_alignment}\ref{lem:ast_to_alignment_a} holds. Therefore, we may assume that $|Y|<a$. Let $W=\{u_y:y\in Y\}$; then we have $|W|<a$. We claim that:
  
  \sta{\label{st:minimality} For every $x\in S_0$, we have $L_x\cap W\neq\varnothing$.}

Suppose for a contradiction that for some $x\in S_0$, we have $L_x\cap W=\varnothing$. Then $x\notin Y$. By the maximality of $Y$, there exists $y\in Y$ such that $L_x\cap L_y\neq \varnothing$. Since $L_x\cap W=\varnothing$, it follows that $u_y\notin L_x$ (and so $u_x\in L_y$), and $L_x\cap L_{y'}=\varnothing$ for all $y'\in Y\setminus \{y\}$. In particular, we have $|u\dd L_0\dd u_x|>|u\dd L_0\dd u_y|$, and the paths $\{L_{y'}:y'\in (Y\setminus \{y\})\cup \{x\}\}$ are pairwise disjoint. But now $(Y\setminus \{y\})\cup \{x\}$ is a better choice than $Y$, a contradiction. This proves \eqref{st:minimality}.
\medskip
  
Since $|S_0|=a^{l-1}(s+(l-1)d)$ and $|W|<a$, it follows from \eqref{st:minimality} that there exists $A\subseteq S_0$ with $|A|=a^{l-2}(s+d(l-1))$ and a vertex $w\in W$, such that for every $x\in A$, we have $w\in L_x$. On the other hand, since $(S_0,L_0)$ is $d$-meager, there are at most $d$ vertices in $A$ which are adjacent to $w$, and so 
  $$|A\setminus N_A(w)|\geq a^{l-2}(s+d(l-1))-d\geq a^{l-2}(s+d(l-2))>0.$$
  It follows that $L_0\setminus \{w\}$ has two components, say $L_1$ and $L_2$, and there exists $B\subseteq A$ with $|B|=a^{l-2}(s+d(l-2))$ such that every vertex in $B$ has a neighbor in $L_1$ and a neighbor in $L_2$. It follows that $(B,L_1)$ and $(B,L_2)$ are both $d$-meager $(a^{l-2}(s+d(l-2)),1)$-constellations in $G$. From the induction hypothesis applied to $(B,L_1)$, we deduce that either there exists an $a$-alignment $(S,L)$ in $G$ with $S\subseteq B\subseteq S_0$ and $L\subseteq L_1\subseteq L_0$, or there exists a plain $(s,l-1)$-constellation $\mf{c}_1$ in $G$ such that $S_{\mf{c}_1}\subseteq B\subseteq S_0$ and $L\subseteq L_1\subseteq L_0$ for every $L\in \mca{L}_{\mf{c}_1}$. In the former case, Lemma~\ref{lem:ast_to_alignment}\ref{lem:ast_to_alignment_a} holds, as required. In the latter case, $\mf{c}=(S_{\mf{c}_1}, \mca{L}_{\mf{c}_1}\cup \{L_2\})$ is a plain $(s,l)$-constellation in $G$ such that $S_{\mf{c}}=S_{\mf{c}_1}\subseteq S_0$ and $L\subseteq L_1\cup L_2\subseteq L_0$ for every $L\in \mca{L}_{\mf{c}}$, and so Lemma~\ref{lem:ast_to_alignment}\ref{lem:ast_to_alignment_b} holds. This completes the proof of Lemma~\ref{lem:ast_to_alignment}.
  \end{proof}

Indeed, for pinched graphs, it turns out that Lemma~\ref{lem:ast_to_alignment}\ref{lem:ast_to_alignment_a} is the only possible outcome:

\begin{lemma}\label{lem:pinchedalignment}
     Let $a,c,d,h\in \poi$ and let $G$ be a $(c,h)$-pinched graph. Assume that there exists a $d$-meager $(a^{2cdh-1}d(h+2cdh-1),1)$-constellation $(S_0,L_0)$ in $G$. Then there exists an $a$-alignment $(S,L,\pi)$ in $G$ with $S\subseteq S_{0}$ and $L\subseteq L_{0}$.
\end{lemma}
\begin{proof}
    Suppose not. Then applying Lemma~\ref{lem:ast_to_alignment} to $(S_0,L_0)$, it follows that there exists a plain $(dh,2cdh)$-constellation  $\mf{c}$ in $G$ such that $S_{\mf{c}}\subseteq S_{0}$ and $L\subseteq L_{0}$ for every $L\in \mca{L}_{\mf{c}}$. In particular, $\mf{c}$ is $d$-meager. Let $u$ be an end of $L_{0}$. Let $\mca{L}_{\mf{c}}=\{L_i:i\in \poi_{2cdh}\}$ and let $u_i,v_i$ be the ends of $L_i$ for each $i\in \poi_{2cdh}$, such that traversing $L_0$ starting at $u$, the vertices $u_1,v_1,\ldots, u_{2cdh},v_{2cdh}$ appear on $L_{0}$ in this order. For each $i\in \poi_{2cdh-1}$, let $v'_i$ be the neighbor of $v_i$ in $L_0\setminus L_i$.  Since $\mf{c}$ is plain, it follows that $v_i\dd L_0\dd u_{i+1}$ is a path of length at least two in $L_0$ whose interior is disjoint from $\{u_i:i\in \poi_{2cdh}\}\cup \{v_i:i\in \poi_{2cdh}\}$ and contains $v'_i$.
    
    For every $i\in \poi_{2cdh}$, since $\mf{c}$ is a constellation, every vertex in $S_{\mf{c}}$ has a neighbor in $L_i$. Let $R_i$ be the shortest path in $L_i$ containing $v_i$ such that every vertex in $S_{\mf{c}}$ has a neighbor in $R_i$. Since $\mf{c}$ is $d$-meager and $|S_{\mf{c}}|=dh$, it follows that $|R_i|\geq h$ for all $i\in \poi_{2cdh}$. Let $w_i$ be the end of $R_i$ distinct from $v_i$. Then the minimality of $R_i$ implies that there exists a vertex $x_i\in S_{\mf{c}}$ that is adjacent to $w_i$ and anticomplete to $R_i\setminus \{w_i\}$. Since $|S_{\mf{c}}|=dh$, it follows that there exist $x\in S_0$ as well as two disjoint $c$-subsets $\{i_k:k\in \poi_c\}$ and $\{j_k:k\in \poi_c\}$ of $\poi_{2cdh}$, such that $i_1<j_1<\cdots<i_{c}<j_{c}$ and we have $x_{i_k}=x_{j_k}=x$ for all $k\in \poi_c$.
    
    For each $k \in \poi_c$, traversing $v'_{i_{k}}\dd L_0\dd w_{j_k}$ starting at $v'_{i_{k}}$, let $w'_k$ be the first neighbor of $x$ in $v'_{i_{k}}\dd L_0\dd w_{j_k}$. Since $x$ is adjacent to $w_{j_k}$, it follows that $w'_k$ exists (See Figure~\ref{fig:constpinch}). Let $C_k = x\dd  w_{i_k} \dd L_{0} \dd w'_k \dd x$. Then $C_k$ is a cycle of length at least $h + 2$ in $G$, as $R_k \cup \{v'_{i_k}\} \subseteq C_k$ and $|R_k| \geq h$. 
  \begin{figure}[t!]
  \centering
  \includegraphics[scale=0.7]{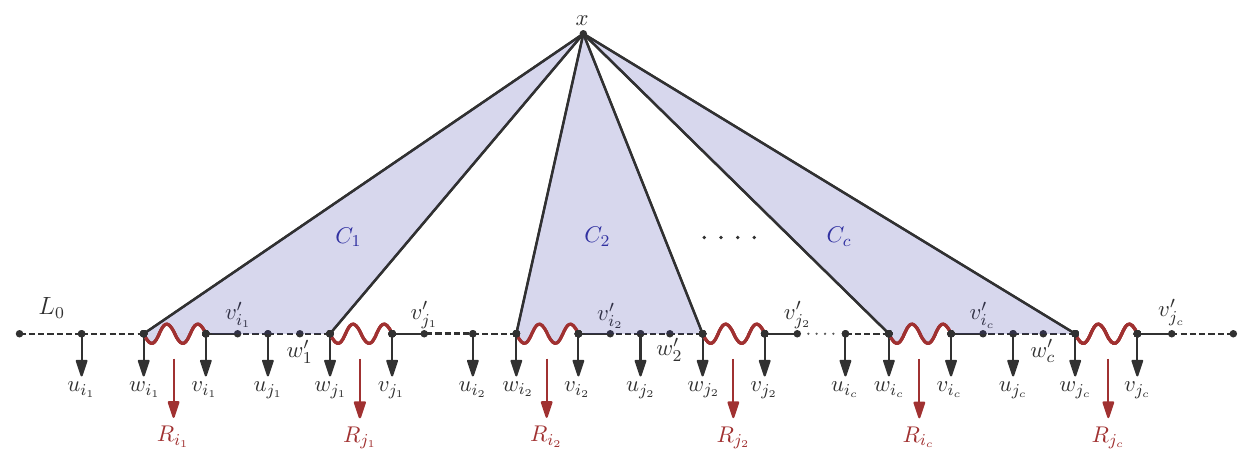}
  \captionof{figure}{Proof of Lemma~\ref{lem:pinchedalignment} (dashed lines depict paths of arbitrary length, and highlighted paths have length at least $h$).}
  \label{fig:constpinch}
  \end{figure}
 Now $C_1,\ldots, C_c$ are $c$ cycles of length at least $h+2$ in $G$ with $C_1\cap\cdots\cap C_c=\{x\}$. Also, since $C_1\setminus \{x\}, \ldots, C_c\setminus \{x\}\subseteq L$ are contained in pairwise distinct component of $L_{0} \setminus \{v'_{j_k}: k \in \poi_c\}$, it follows that $C_1\setminus \{x\}, \ldots, C_c\setminus \{x\}\subseteq L$ are pairwise disjoint and anticomplete in $G$. This violates the assumption that $G$ is $(c,h)$-pinched, hence completing the proof of Lemma~\ref{lem:pinchedalignment}.
\end{proof}

We also need the following quantified version of Ramsey's Theorem. This has appeared in several references; see, for instance, \cite{eta}.

\begin{theorem}[Ramsey \cite{multiramsey}, see also \cite{eta}]\label{classicalramsey}
For all $c,s\in \poi$, every graph $G$ on at least $c^s$ vertices contains either a clique of cardinality $c$ or a stable set of cardinality $s$. 
\end{theorem}

From Theorem~\ref{classicalramsey}, we deduce that:

\begin{lemma}\label{lem:meagerconst}
  Let $l,s,t\in \poi$, let $G$ be a graph and let $\mf{c}$ be a $(s,l+(st)^{t})$-constellation in $G$. Then one of the following holds.
\begin{enumerate}[\rm (a)]
    \item\label{lem:meagerconst_a} $G$ contains $K_t$ or $K_{t,t}$. 
        \item\label{lem:meagerconst_b} There exists $\mca{L}\subseteq 
    \mca{L}_{\mf{c}}$ with $|\mca{L}|=l$ such that $(S_{\mf{c}},\mca{L})$ is a $t$-meager $(s,l)$-constellation in $G$.
    \end{enumerate}
\end{lemma}
\begin{proof}
   Suppose that Lemma~\ref{lem:meagerconst}\ref{lem:meagerconst_a} does not hold. For every $L\in \mca{L}_{\mf{c}}$, let $t_{L}$ be the largest number in $\poi_s$ for which some vertex $u_L\in L$ has at least $t_{L}$ neighbors in $S_{\mf{c}}$. It follows that:

    \sta{\label{st:fewtL} We have $|\{L\in \mca{L}_{\mf{c}}:t_L\geq t\}|<(st)^{t}$.}

    Suppose not. Let $\mca{S}\subseteq \{L\in \mca{L}_{\mf{c}}:t_L\geq t\}$ with $|\mca{S}|=(st)^{t}$. Then for every $L\in \mca{S}$, we may choose $u_L\in L$ and $T_L\subseteq S_{\mf{c}}$ such that $|T_L|=t$ and $u_L$ is complete to $T_L$ in $G$. Since $|S_{\mf{c}}|=s$, it follows that there exist $T\subseteq S_{\mf{c}}$ and $\mca{T}\subseteq \mca{S}$ with $|T|=t$ and $|\mca{T}|=t^t$,  such that for every $L\in \mca{T}$, we have $T_L=T$. Let $U=\{u_L:L\in \mca{T}\}$. Then $T$ and $U$ are disjoint and complete in $G$. Also, since $G$ is $K_t$-free and $|U|=t^t$, it follows from Theorem~\ref{classicalramsey} that there is a stable set $U'\subseteq U$ in $G$ with $|U'|=t$. But then $G[T\cup U']$ is isomorphic to $K_{t,t}$, a contradiction. This proves \eqref{st:fewtL}.

    \medskip

     Now the result is immediate from \eqref{st:fewtL} and the fact that $|\mca{L}_{\mf{c}}|=l+t^ts^{t}$. This completes the proof of Lemma~\ref{lem:meagerconst}. 
    \end{proof}

We are now ready to prove Theorem~\ref{thm:dealwithconst}.

\begin{proof}[Proof of Theorem~\ref{thm:dealwithconst}]
We claim that:
$$\sigma=\sigma(c,h,s,t)=s^{2cht-1}t(h+2cht-1);$$
$$\lambda=\lambda(c,h,s,t)=cs(s!)\sigma^s+(\sigma t)^t$$
satisfy Theorem~\ref{thm:dealwithconst}. Let $G$ be a $t$-clean $(c,h)$-pinched graph which does not contain $K_t$ and $K_{t,t}$. Let $\mf{c}$ be a plain $(\sigma,\lambda)$-constellation in $G$. By the choice of $\sigma$ and $\lambda$, we may apply Lemma~\ref{lem:meagerconst} to $\mf{c}$, and deduce that Lemma~\ref{lem:meagerconst}\ref{lem:meagerconst_b} holds, that is, there exists $\mca{L}_0\subseteq \mca{L}_\mf{c}$ with $|\mca{L}_0|=cs(s!)\sigma^s$ such that for every $L\in \mca{L}_0$, $(S_{\mf{c}},L)$ is a $t$-meager, plain $(\sigma,1)$-constellation in $G$. In particular, since $\sigma=s^{2cht-1}t(h+2cht-1)$,  we may apply Lemma~\ref{lem:pinchedalignment} to $(S_{\mf{c}},L)$ to show that:

\sta{\label{st:getalignment} For every $L\in \mca{L}_0$, there exists an $s$-alignment $(S_L,Q_L,\pi_L)$ in $G$ with $S_{L}\subseteq S_{\mf{c}}$ and $Q_{L}\subseteq L$.}

Note that $|S_{L}|=s$ for all $L\in \mca{L}_0$. Recall also that $|S_{\mf{c}}|=\sigma$ and $|\mca{L}_0|=cs(s!)\sigma^s$. This, along with a pigeon-hole argument, implies immediately that:

\sta{\label{st:almostarray}There exists $S\subseteq S_{\mf{c}}$ with $|S|=s$, and $\mca{L}\subseteq \mca{L}_0$ with $|\mca{L}|=cs$ and $\pi:\poi_s\rightarrow S$ such that for every $L\in \mca{L}$, $(S,Q_L,\pi)$ is an $s$-alignment in $G$.}

We further deduce that:

\sta{\label{st:usehollow}There exists $\mca{S}\subseteq \mca{L}$ with $|\mca{S}|=s$ such that for every $L\in \mca{S}$, the $(s,1)$-constellation $(S,Q_L)$ is $h$-hollow.}

To see this, since $|S|=s$ and $|\mca{L}|=cs$, it suffices to show that for every $x\in S$, there are fewer than $c$ paths $L\in \mca{L}$ for which there is an $x$-gap in $(S,Q_L)$ of length at least $h$. Suppose for a contradiction that for some $x\in S$, there are $c$ distinct paths $L_1,\ldots, L_c\in \mca{L}$ such that for every $i\in \poi_c$, there is an $x$-gap $P_i$ in $(S,Q_{L_i})$ of length at least $h$; let $y_i,z_i$ be the ends of $P_i$. It follows that $C_i=x\dd y_i\dd P_i\dd z_i\dd x$ is a cycle of length at least $h+2$ in $G$. Now $C_1,\ldots, C_c$ are $c$ cycles of length at least $h+2$ in $G$ with $C_1\cap\cdots\cap C_c=\{x\}$. Moreover, since $\mf{c}$ is plain, $C_1\setminus \{x\}, \ldots, C_c\setminus \{x\}\subseteq V(\mathcal{L})$ are pairwise disjoint and anticomplete in $G$. This violates the assumption that $G$ is $(c,h)$-pinched, and so proves \eqref{st:usehollow}.

\medskip

Let $\mf{a}=(S,\mca{S})$ where $S$ comes from \eqref{st:almostarray} and $\mca{S}$ comes from \eqref{st:usehollow}. Also, let $\pi$ be as in \eqref{st:almostarray}. It follows from \eqref{st:almostarray} and \eqref{st:usehollow} that $\mf{a},\pi$ satisfies \ref{AR}. Hence, $\mf{a}$ is an $(s,h)$-array in $G$. This completes the proof of Theorem~\ref{thm:dealwithconst}.
\end{proof}

\section{Obtaining a plain constellation}\label{sec:obtain}
This section contains the main ingredient of the proof of Theorem~\ref{mainthmgeneral}:    
\begin{theorem}\label{thm:block_to_constellation}
For all $c,h,l,s,t\in \poi$, there is a constant $
       \Omega=\Omega(c,h,l,s,t)\in \poi$ with the following property. Let $G$ be a $(c,h)$-pinched graph. Assume that there exists a strong $\Omega$-block in $G$. Then one of the following holds.
       \begin{enumerate}[\rm (a)]
    \item\label{thm:block_to_constellation_a} $G$ contains either $K_t$ or $K_{t,t}$.
    \item\label{thm:block_to_constellation_b} There exists a plain $(s,l)$-constellation in $G$.
      \end{enumerate}
  \end{theorem}
Our road to the proof of Theorem~\ref{thm:block_to_constellation} passes through a number of definitions and lemmas from \cite{twix}, beginning with two useful Ramsey-type results for polypaths:

\begin{lemma}[Alecu, Chudnovsky, Hajebi, Spirkl; see Lemma 5.2 in  \cite{twix}]\label{lem:bundle}
   \sloppy For all $b,f,g,m,n,t\in \poi$ and $s\in \poi\cup \{0\}$, there is a constant $\beta=\beta(b,f,g,m,n,s,t)\in \poi$ with the following property. Let $G$ be a graph and let $\mf{B}$ be a collection of $\beta$ pairwise disentangled $(b,2b(g-1)+f)$-bundles in $G$. Then one of the following holds.
    \begin{enumerate}[\rm (a)]
        \item\label{lem:bundle_a} $G$ contains either $K_t$ or $K_{t,t}$.
        \item\label{lem:bundle_b} There exists  $\mf{N}\subseteq \mf{B}$ with $|\mf{N}|=n$ as well as $S\subseteq \bigcup_{\mf{b}\in \mf{B}\setminus \mf{N}} S_{\mf{b}}$ with $|S|=s$, such that for every $\mf{b}\in \mf{N}$, there exists $\mca{G}_{\mf{b}}\subseteq \mca{L}_{\mf{b}}$ with $|\mca{G}_{\mf{b}}|=g$ for which $(S, \mca{G}_{\mf{b}})$ is an $(s,g)$-constellation in $G$.
        \item\label{lem:bundle_c} There exists  $\mf{M}\subseteq \mf{B}$ with $|\mf{M}|=m$ as well as $\mca{F}_{\mf{b}}\subseteq \mca{L}_{\mf{b}}$ with $|\mca{F}_{\mf{b}}|=f$ for each $\mf{b}\in \mf{M}$, such that for all distinct $\mf{b},\mf{b}'\in \mf{M}$, $S_{\mf{b}}$ is anticomplete to $S_{\mf{b}'}\cup V(\mca{F}_{\mf{b}'})$ in $G$.
    \end{enumerate}
\end{lemma}

\begin{lemma}[Alecu, Chudnovsky, Hajebi, Spirkl; see Lemma 5.3 in  \cite{twix}]\label{lem:polypathvspolypath}
    For all $a,g\in \poi$, there is a constant $\varphi=\varphi(a,g)\in \poi$ with the following property. Let $G$ be a graph and let $\mca{F}_1,\ldots, \mca{F}_a$ be a collection of $a$ pairwise disentangled $\varphi$-polypaths in $G$. Then for every $i\in \poi_a$,  there exists a $g$-polypath $\mca{G}_{i}\subseteq \mca{F}_i$, such that for all distinct $i,i'\in \poi_a$, either $V(\mca{G}_i)$ is anticomplete to $V(\mca{G}_{i'})$ in $G$, or for every $L\in \mca{G}_i$ and every $L'\in \mca{G}_{i'}$, there is an edge in $G$ with an end in $L$ and an end in $L'$.
\end{lemma}

We continue with a few definitions from \cite{twix}. Let $G$ be a graph, let $d\in \poi\cup \{0\}$ and let $r\in \poi$. For $X\subseteq V(G)$, by a \textit{$(d,r)$-patch for $X$ in $G$} we mean a $(1,r)$-bundle $\mf{p}$ in $G$ where:
\begin{enumerate}[(P1), leftmargin=19mm, rightmargin=7mm]
     \item\label{P1} $S_{\mf{p}}\subseteq V(G)\setminus V(\mca{L}_{\mf{p}})$;
    \item\label{P2} every path $L\in \mca{L}_{\mf{p}}$ has length at least $d$; and
    \item\label{P3} for every $L\in \mca{L}_{\mf{p}}$, one may write $\partial L=\{x_L,y_L\}$ such that $L\cap X=\{x_L\}$ and $N_L(S_\mf{p})=\{y_L\}$.
 \end{enumerate}
 Also, by a \textit{$(d,r)$-match for $X$ in $G$} we mean an $r$-polypath $\mca{M}$ in $G$ such that 

 \begin{enumerate}[(M1), leftmargin=19mm, rightmargin=7mm]
 \item\label{M1} every path $L\in \mca{M}$ has length at least $d$; and
 \item\label{M2} $V(\mca{M})\cap X=\partial \mca{M}$.
 \end{enumerate}
 \begin{figure}[t!]
  \centering
  \includegraphics[scale=0.8]{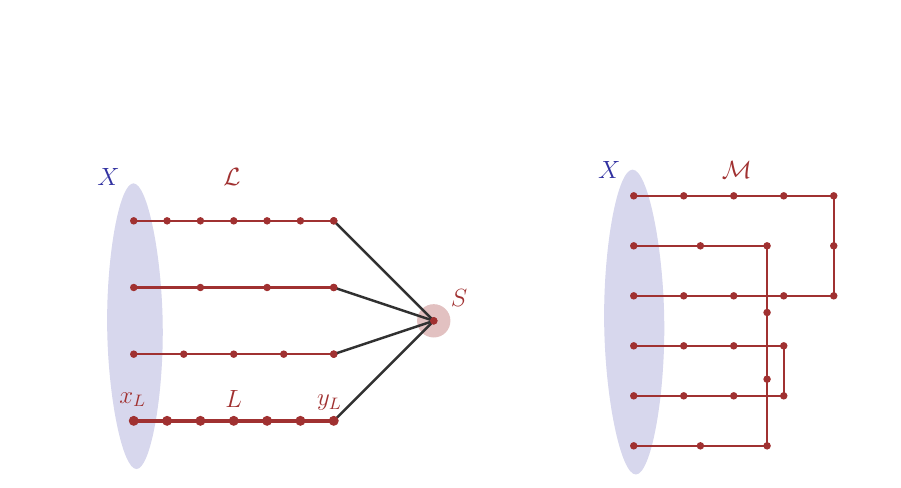}
  \captionof{figure}{A $(3,4)$-patch for $X$ (left) and a $(7,3)$-match for $X$ (right).}
  \label{fig:patch&match}
  \end{figure}
See Figure~\ref{fig:patch&match}. We also need:
 \begin{lemma}[Alecu, Chudnovsky, Hajebi, Spirkl; see Lemma 7.2 in \cite{twix}]\label{lem:digraph}
       \sloppy For all $d,l,m,r,r',s\in \poi$, there is a constant $
       \eta=\eta(d,l,m,r,r',s)\in \poi$ with the following property. Let $G$ be a graph, let $X\subseteq V(G)$ and let $\mf{p}$ be a $(d,\eta)$-patch for $X$ in $G$. Then one of the following holds.
    \begin{enumerate}[\rm (a)]
        \item\label{lem:digraph_a} There exists a plain $(s,l)$-constellation in $G$.
         \item\label{lem:digraph_b} $G$ contains a $(\leq d+1)$-subdivision of $K_m$ as a subgraph.
         \item\label{lem:digraph_c} There exists a plain $(2(d+1),r)$-match $\mca{M}$ for $X$ in $G$ such that $ V(\mca{M})\subseteq V(\mf{p})$.
        \item\label{lem:digraph_d} There exists a plain $(d,r')$-patch $\mf{q}$ for $X$ in $G$ such that $V(\mf{q}) \subseteq V(\mf{p})$.
    \end{enumerate}
  \end{lemma}

We can now prove the main result of this section:

\begin{proof}[Proof of Theorem~\ref{thm:block_to_constellation}]
 Let $H$ be the unique graph (up to isomorphism) consisting of $c$ induced cycles of length $h+2$, all sharing a common vertex and otherwise pairwise disjoint and anticomplete. Let $m=m(H,h-1,t)$ be as in Theorem~\ref{dvorak} (note that $m$ only depends on $c,h,t$).

 Let $\gamma=(c(h+2))^{2chl-1}l(h+2chl-1)$; then $\gamma>1$ (in fact, it holds that $\gamma\geq 6$). Let
 $\varphi=\varphi(c,s(\gamma+2l-2)^l)$
 be as in Lemma~\ref{lem:polypathvspolypath}. Let
$$\beta=\beta(1,\varphi,l,c,1,s,t)$$
be as in Lemma~\ref{lem:bundle}, and let 
$$\eta=\eta(h-1,l,m,c,2(l-1)+\varphi,s)$$
be as in Lemma~\ref{lem:digraph}. We prove that
$$\Omega=\Omega(c,h,l,s,t)=\max\{m^{\beta+1},\eta\}$$
satisfies Theorem~\ref{thm:block_to_constellation}. 

Let $G$ be a $(c,h)$-pinched graph such that there is a strong $\Omega$-block in $G$. Then $G$ does not contain a subdivision of $H$. Suppose for a contradiction that $G$ contains neither $K_t$ nor $K_{t,t}$, and there is no plain $(s,l)$-constellation in $G$. By Theorem~\ref{dvorak} and the choice of $m$, $G$ contains no subgraph isomorphic to a $(\leq h-1)$-subdivision of $K_m$. It is convenient to sum up all this in one statement:

\sta{\label{st:blockcontradiction} The following hold.
\begin{itemize}
    \item $G$ does not contain $K_t$ or $K_{t,t}$.
    \item There is no plain $(s,l)$-constellation in $G$.
     \item $G$ does not contain a subgraph isomorphic to a $(\leq h-1)$-subdivision of $K_m$.
\end{itemize}}

Let $B$ be a strong $\Omega$-block in $G$ and for every $2$-subset $\{x,y\}$ of $B$, let $\mca{P}_{\{x,y\}}$ be the corresponding set of $\Omega$ paths in $G$ from $x$ to $y$. Let $J$ be a graph with vertex set $B$ such that two distinct vertices $x,y\in B$ are adjacent in $J$ if and only if there exists $P_{\{x,y\}}\in \mca{P}_{\{x,y\}}$ of length at most $h$. Since $|V(J)|=|B|\geq m^{\beta+1}$, it follows from Theorem~\ref{classicalramsey}  that $J$ contains either clique $C$ on $m$ vertices, or a stable set on $\beta+1$ vertices. In the former case, the union of the paths $P_{\{x,y\}}$ for all distinct $x,y\in C$ forms a subgraph of $G$ isomorphic to a $(\leq h-1)$-subdivision of $K_m$, which violates the third bullet in \eqref{st:blockcontradiction}. Therefore, $J$ contains a stable set on $\beta+1$ vertices. This, along with the choice of $\Omega$, implies that one may choose $\beta+1$ distinct vertices $x,y_1,\ldots, y_{\beta}\in V(G)$ as well as, for each $i\in \poi_{\beta}$, a collection $\mca{P}_i$ of $\eta$ pairwise internally disjoint paths in $G$ from $x$ to $y_i$, such that:
\begin{itemize}
    \item for every $i\in \poi_{\beta}$, every path $P\in \mca{P}_i$ has length at least $h+1\geq 2$; and
    \item $V(\mca{P}_1)\setminus \{x\},\ldots, V(\mca{P}_{\beta})\setminus \{x\}$ are pairwise disjoint in $G$.
\end{itemize}

For each $i\in \poi_{\beta}$, let $\mca{L}_i=\mca{P}^*_i$, and for every $L\in \mca{L}_i$, let $x_L$ and $y_L$ be the (unique) neighbors of $x$ and $y_i$ in $L$, respectively. Thus, we have $\partial L=\{x_L,y_L\}$ (note that $x_L,y_L$ might be the same, but they are distinct from both $x$ and $y_i$). Let $X_i=\{x_L:L\in \mca{L}_i\}$. It follows that $\mca{L}_i$ is an $\eta$-polypath in $G$, and
\begin{itemize}
    \item $y_i\in V(G)\setminus V(\mca{L}_i)$;
    \item every path $L\in \mca{L}_i$ has length at least $h-1$; and
    \item $L\cap X_i=\{x_L\}$ and $N_L(y_i)=\{y_L\}$ for every $L\in \mca{L}_i$.
\end{itemize}
Therefore, by \ref{P1}, \ref{P2} and \ref{P3}, for every $i\in \poi_{\beta}$, the $(1,\eta)$-bundle $\mf{p}_i=(\{y_i\},\mca{L}_i)$ is an $(h-1,\eta)$-patch for $X_i$ in $G$. In addition, we show that:

\sta{\label{st:applypatchlemme} For every $i\in \poi_{\beta}$, there is a plain $(h-1,2(l-1)+\varphi)$-patch $\mf{q}_i$ for $X_i$ in $G$ with $V(\mf{q}_i) \subseteq V(\mf{p}_i)\subseteq V(G)\setminus \{x\}$.}

By the choice of $\eta$, we can apply Lemma~\ref{lem:digraph} to $X_i$ and $\mf{p}_i$. Note that Lemma~\ref{lem:digraph}\ref{lem:digraph_a} and Lemma~\ref{lem:digraph}\ref{lem:digraph_b} violate the second and the third bullets of \eqref{st:blockcontradiction}, respectively. Assume that Lemma~\ref{lem:digraph}\ref{lem:digraph_c} holds. Then there is a plain $(2h,c)$-match $\mca{M}_i$ for $X_i$ in $G$ with $ V(\mca{M}_i)\subseteq V(\mf{p}_i)\subseteq V(G)\setminus \{y_i\}$. In particular, for every $M\in \mca{M}_i$, $C_M=M\cup \{x\}$ is a cycle of length at least $2h+2$ in $G$. But now $\{C_M:M\in \mca{M}_i\}$ is a collection of $c$ cycles of length at least $2h+2$ in $G$ where $\bigcap_{M\in \mca{M}_i}C_M=\{x\}$ and the sets $\{C_M\setminus \{x\}:M\in \mca{M}\}$ are pairwise disjoint and anticomplete in $G$, violating the assumption that $G$ is $(c,h)$-pinched. So Lemma~\ref{lem:digraph}\ref{lem:digraph_d} holds. This proves \eqref{st:applypatchlemme}.

\medskip

For each $i\in \poi_{\beta}$, let $\mf{q}_i$ be as in \eqref{st:applypatchlemme} and let $S_{\mf{q}_i}=\{z_i\}$ (note that $z_i$ may or may not belong to $X_i\subseteq N_G(x)$). Now, $\mf{B}=\{\mf{q}_1,\ldots, \mf{q}_{\beta}\}$ is a collection of $\beta$ pairwise disentangled plain $(1,2(l-1)+\varphi)$-bundles in $G$. Given the choice of $\beta$, we can apply Lemma~\ref{lem:bundle} to $\mf{B}$. Note that Lemma~\ref{lem:bundle}\ref{lem:bundle_a} directly violates  the first bullet of \eqref{st:blockcontradiction}. Also, if Lemma~\ref{lem:bundle}\ref{lem:bundle_b} holds, then there is an $(s, l)$-constellation $\mf{c}$ in $G$ with $\mca{L}_{\mf{c}}\subseteq \mca L_{\mf{q}_i}$ for some $i \in \poi_{\beta}$; in particular, $\mf{c}$ is plain. But this violates the second bullet of \eqref{st:blockcontradiction}.  It follows that  Lemma~\ref{lem:bundle}\ref{lem:bundle_c} holds, that is, there exists $I\subseteq \poi_{\beta}$ with $|I|=c$ as well as $\mca{F}_{i}\subseteq \mca{L}_{\mf{q}_i}$ with $|\mca{F}_{i}|=\varphi$ for each $i\in I$, such that for all distinct $i,i'\in I$, $z_i$ is anticomplete to $V(\mca{F}_{i'})\cup \{z_{i'}\}$ in $G$.  We further deduce that:

\sta{\label{st:polyvspoly} There are distinct $i,i'\in I$ for which there exist $s(\gamma+2l-2)^l$-polypaths $\mca{G}_{i}\subseteq \mca{F}_i$ and $\mca{G}_{i'}\subseteq \mca{F}_{i'}$ such that for every $L\in \mca{G}_i$ and every $L'\in \mca{G}_{i'}$, there is an edge in $G$ with an end in $L$ and an end in $L'$.}

Suppose not. Then by the choice of $\varphi$, we may apply Lemma~\ref{lem:polypathvspolypath} to the $\varphi$-polypaths $\{\mca{F}_i:i\in I\}$ and deduce that for every $i\in I$,  there exists a $s(\gamma+2l-2)^l$-polypath $\mca{G}_{i}\subseteq \mca{F}_i$, such that for all distinct $i,i'\in \poi_a$, the sets $V(\mca{G}_i)$ and $V(\mca{G}_{i'})$ are anticomplete in $G$. Since $s(\gamma+2l-2)^l\geq \gamma>1$, it follows that for each $i\in I$, there is a cycle $C_i$ in $G[V(\mca{G}_i)\cup \{x,z_i\}]$ of length at least $h+2$ where $x\in C_i$. Also, the sets $\{V(\mca{G}_i)\cup \{z_i\}:i\in I\}$ are pairwise disjoint and anticomplete in $G$, which in turn implies that the sets $\{C_i\setminus \{x\}:i\in I\}$ are pairwise disjoint and anticomplete in $G$. This yields a contradiction with the assumption that $G$ is $(c,h)$-pinched, and so proves \eqref{st:polyvspoly}.

\medskip

Henceforth, let $i,i'\in I$ and $\mca{G}_i,\mca{G}_{i'}$ be as in \eqref{st:polyvspoly}. Since $|\mca{G}_{i}|=|\mca{G}_{i'}|= s(\gamma+2l-2)^l\geq \gamma+2l-2$, we may choose  $\mca{G}'\subseteq \mca{G}_{i'}$ with $|\mca{G}'|=\gamma+2l-2$. Write $\mca{G}= \mca{G}_{i}$. It follows that both $\mca{G}$ and $\mca{G}'$ are plain and disentangled polypaths in $G$. For every path $L\in \mca{G}$, let us say $L$ is \textit{rigid} if there exists a vertex $x_L\in L$ as well as $\mca{G}'_L\subseteq \mca{G}'$ with $|\mca{G}'_L|=l$ such that $x_L$ has a neighbor in every path in $\mca{G}'_L$. 
We claim that:

\sta{\label{st:fewrigidpaths} 
        Not all paths in  $\mca{G}$ are rigid, that is, there is a path $L_0\in \mca{G}$ such that no vertex in $L_0$ has neighbors in more than $l-1$ paths in $\mca{G}'$.}

Suppose not. For every $L\in \mca{G}$, let $x_L\in L$ and $\mca{G}'_L\subseteq \mca{G}'$ with $|\mca{G}'_L|=l$ be as in the definition of a rigid path. Since $|\mca{G}|=s(\gamma+2l-2)^l$ and $|\mca{G}'|=\gamma+2l-2$, it follows that there exists $\mca{S}\subseteq \mca{G}$ with $|\mca{S}|=s$ and $\mca{L}\subseteq \mca{G}'$ with $|\mca{L}|=l$ such that for every $L\in \mca{S}$, we have $\mca{G}'_L=\mca{L}$. Let $S=\{x_L:L\in \mca{S}\}$. Then every vertex in $S$ has a neighbor in every path in $\mca{L}$. But now since $\mca{G}$ and $\mca{G}'$ are plain and disentangled, it follows that $(S,\mca{L})$ is a plain $(s,l)$-constellation in $G$, a contradiction with the second bullet of \eqref{st:blockcontradiction}. This proves \eqref{st:fewrigidpaths}.
\medskip

By \eqref{st:fewrigidpaths}, there exists a path $L_0\in \mca{G}$ that is not rigid. Let $\mca{T}$ be the set of all paths $L'\in \mca{G}'$ for which some vertex in $\partial L_0$ has a neighbor in $L'$ in $G$. Then by \eqref{st:fewrigidpaths}, we have $|\mca{T}|\leq 2l-2$. Consequently, there are $\gamma$ distinct paths $L'_1, \ldots, L'_{\gamma}\in \mca{G}'$ such that for every $i\in \poi_{\gamma}$, the ends of $L_0$ are anticomplete to $L'_i$ in $G$. For each $i\in \poi_{\gamma}$, let $x_i$ be the end of $L'_i$ which is adjacent to $x$. By \eqref{st:polyvspoly}, traversing $L'_i$ starting at $x_i$, we may choose $x'_i\in L'_i$ to be the first vertex in $L'_i$ with a neighbor in $L_0^*$. Let $S_0=\{x'_i:i\in \poi_{\gamma}\}$. It follows that every vertex in $S_0$ has a neighbor in $L_0$ while $S_0$ is anticomplete to the ends of $L_0$, and so $(S_0,L_0)$ is a $(\gamma,1)$-constellation in $G$. Also, by \eqref{st:fewrigidpaths}, $(S_0,L_0)$ is $(l-1)$-meager. This calls for an application of Lemma~\ref{lem:pinchedalignment} to $(S_0,L_0)$, which implies that there exists a $c(h+2)$-alignment $(S,L,\pi)$ in $G$ with $S\subseteq S_0$ and $L\subseteq L_0$. In fact, since $S_0$ is anticomplete to $\partial L_0$, we may assume that $L\subseteq L_0^*$.

Let $u$ be the end of $L$ for which $(S,L,\pi)$ satisfies \ref{AL}. For every $i\in \poi_c$, choose $\varepsilon_i,\delta_i,\theta_i\in \poi_{\gamma}$ such that
$$x'_{\varepsilon_i}=\pi(i(h+2)-h-1);$$
$$x'_{\delta_i}=\pi(i(h+2)-1);$$
$$x'_{\theta_i}=\pi(i(h+2));$$
Traversing $L$ starting at $u$, let $u_i$ be the last neighbor of $x'_{\varepsilon_i}$ in $L$, let $v_i$ be the first neighbor of  $x'_{\delta_i}$ in $L$, and let $w_i$ be the first neighbor of  $x'_{\theta_i}$ in $L$. It follows that $u,u_1,v_1,w_1\ldots, u_c,v_c,w_c$ are pairwise distinct, appearing on $L$ in this order. Also, for every $i\in \poi_c$, $P_i=u_i\dd L\dd v_i$ is a path of length at least $h$ in $G$, and $P_1,\ldots, P_c$ are contained in distinct components of $L\setminus \{w_i:i\in \poi_c\}$. Thus, $P_1,\ldots, P_c$ are pairwise disjoint and anticomplete in $G$. Now, for every $i\in \poi_c$, consider the cycle
$$C_i=x\dd x_{\varepsilon_i}\dd L'_{\varepsilon_i}\dd x'_{\varepsilon_i}\dd u_i\dd P_i\dd v_i\dd x'_{\delta_i}\dd L'_{\delta_i}\dd x_{\delta_i}\dd x.$$
Then $C_1,\ldots, C_c$ are $c$ cycles in $G$ each of length at least $h+4$, all going through $x$ and otherwise pairwise disjoint and anticomplete. This contradicts the assumption that $G$ is $(c,h)$-pinched, hence completing the proof of Theorem~\ref{thm:block_to_constellation}.
\end{proof}
 
\section{The end}\label{sec:end}
We conclude the paper with the proof of our main result, which we restate:
\setcounter{section}{3}
\setcounter{theorem}{1}
\begin{theorem}
    For all $c,h,s,t\in \poi$, there is a constant $\tau=\tau(c,h,s,t)\in \poi$ such that for every $t$-clean $(c,h)$-pinched graph $G$ of treewidth more than $\tau$, there is an $(s,h)$-array in $G$.
\end{theorem}
\begin{proof}
    Let $\sigma=\sigma(c,h,s,t)$ and     $\lambda=\lambda(c,h,s,t)$ be as in Theorem~\ref{thm:dealwithconst}. Let $\Omega=\Omega(c,h,\lambda,\sigma,t)$ be as in Theorem~\ref{thm:block_to_constellation}. We define  
$\tau(c,h,s,t)=\xi(\Omega,t)$, where $\xi(\cdot,\cdot)$ comes from Theorem~\ref{noblocksmalltw_wall}.  Let $G$ be a $t$-clean $(c,h)$-pinched graph of treewidth more than $\tau$. By Theorem~\ref{noblocksmalltw_wall}, $G$ contains a strong $\Omega$-block. Therefore, since $G$ does not contain $K_t$ and $K_{t,t}$, it follows from Theorem~\ref{thm:block_to_constellation} that there exists a plain $(\sigma,\lambda)$-constellation in $G$. But now by Theorem~\ref{thm:dealwithconst}, there is an $(s,h)$-array in $G$, as desired.
\end{proof}

\setcounter{section}{6}

\section{Acknowledgments}
Our thanks to the anonymous referees for suggesting several improvements.

\bibliographystyle{abbrv}
	\bibliography{ref}

\end{document}